\documentclass[12pt]{amsart}
\usepackage{amssymb,latexsym}
\usepackage{enumerate}
\usepackage{latexsym}
\usepackage{amsthm,amsfonts,amssymb,mathrsfs}
\usepackage{rotating}
\usepackage[leqno]{amsmath}
\usepackage{xspace}
\usepackage[all]{xy}
\usepackage{longtable}

\newtheorem{theorem}{Theorem}[section]
\newtheorem{lemma}[theorem]{Lemma}
\newtheorem{proposition}[theorem]{Proposition}
\newtheorem{corollary}[theorem]{Corollary}

\theoremstyle{definition}
\newtheorem{definition}[theorem]{Definition}
\theoremstyle{remark}

\newtheorem{example}[theorem]{Example}

\newcommand{\tr}{\operatorname{tr}}

\newcommand{\Ker}{\operatorname{Ker}}

\newcommand{\Hom}{\operatorname{Hom}}

\newcommand{\Coker}{\operatorname{Coker}}

\newcommand{\BN}{\Bbb N}

\input xy
\xyoption{all}

\begin{document}

\baselineskip=17pt

\title[The existence of relative pure...]{The Existence of Relative pure Injective Envelopes}

\author[F. Zareh-Khoshchehreh and K. Divaani-Aazar]{Fatemeh Zareh-Khoshchehreh and Kamran Divaani-Aazar (Tehran)}

\address{F. Zareh-Khoshchehreh \\ Department of Mathematics \\ Az-Zahra University \\ 
Vanak Post Code 19834 \\
Tehran, Iran.}
\email{fzarehkh@gmail.com}

\address{K. Divaani-Aazar \\ Department of Mathematics \\ Az-Zahra University \\ Vanak, Post Code 19834 \\ 
Tehran, Iran-and-Institute for Studies in Theoretical Physics and Mathematics \\ P.O. Box 19395-5746 \\ 
Tehran, Iran.}
\email{kdivaani@ipm.ir}

\date{}


\keywords {Auslander transpose; cyclically-presented module; finitely presented module; left global
$\mathcal{S}$-pure injective dimension; left global $\mathcal{S}$-pure projective dimension; right
balanced functor; $\mathcal{S}$-pure flat module; $\mathcal{S}$-pure injective module;
$\mathcal{S}$-pure projective module.\\
The second author was supported by a grant from IPM (No. 91130212)}

\begin{abstract} Let $\mathcal{S}$ be a class of finitely presented $R$-modules such that
$R\in \mathcal{S}$ and $\mathcal{S}$ has a subset $\mathcal{S}^*,$  with the property that
for any $U\in \mathcal{S}$ there is a $U^*\in \mathcal{S}^*$ with $U^*\cong U.$  We show
that the class of $\mathcal{S}$-pure injective $R$-modules is preenveloping. As an application,
we deduce that the left global $\mathcal{S}$-pure projective dimension of $R$ is equal to its
left global $\mathcal{S}$-pure injective dimension. As our main result, we prove that, in fact,
the class of $\mathcal{S}$-pure injective $R$-modules is enveloping.
\end{abstract}

\maketitle

\section{Introduction}

Throughout this paper, $R$ denote a  ring with identity and all modules are assumed to be left and
unitary. The notion of purity plays a substantial role in algebra and model theory. It was introduced
by P.M. Cohn in \cite{C} for left $R$-modules and by J. {\L}o\'{s} in \cite{L} for abelian groups;
see also J.M. Maranda \cite{M}. In 1967, R. Kie{\l}pi\'nski in  \cite{K} has introduced the notion of
relative $\Gamma$-purity and proved that any $R$-module possesses a relative $\Gamma$-pure injective
envelope. Also, he has shown that the relative $\Gamma$-pure injectivity coincides with the relative
$\Gamma$-algebraic compactness. Two years later in \cite{W}, R.B. Warfield has proved that any $R$-module
admits a pure injective envelope and the pure injectivity coincides with the algebraic compactness.
Also, he has introduced a notion of $\mathcal{S}$-purity for any class $\mathcal{S}$ of $R$-modules.
One can check that for an appropriate $\Gamma$, the $\Gamma$-purity and RD-purity coincides.  But, for a
general class $\mathcal{S}$ of finitely presented $R$-modules, the relationship between $\Gamma$-purity and $\mathcal{S}$-purity is ambiguous. For a survey of results on various notions of purity, we refer the
reader to the interesting articles \cite{FSZ}, \cite{G}, \cite{H}-\cite{S2} and \cite{W},  where
among other things the algebraic compactness and pure homological dimensions are discussed.

We call a class $\mathcal{S}$ of $R$-modules \emph{set-presentable} if it has a subset $\mathcal{S}^*,$
with the property that for any $U\in \mathcal{S}$ there is a $U^*\in \mathcal{S}^*$ with $U^*\cong U.$
It is easy to see that any class of finitely presented $R$-modules which is closed under isomorphism is set-presentable. So, the classes of finitely presented $R$-modules,  cyclic cyclically-presented
$R$-modules and cyclically-presented $R$-modules are set-presentable. Also, note that each of these
classes contains $R$. Let $\mathcal{S}$ be a set-presentable class of finitely presented $R$-modules
containing $R$. Warfield \cite[Proposition 1]{W} has shown that every $R$-module possesses an
$\mathcal{S}$-pure projective precover. It is natural to ask whether any $R$-module possesses an
$\mathcal{S}$-pure injective preenvelope. For the set-presentable classes of finitely presented
$R$-modules, cyclic cyclically-presented $R$-modules and cyclically-presented $R$-modules, even more
is proven to be true. Warfield \cite[Proposition 6]{W} has proved that every $R$-module possesses a pure
injective envelope. Also, he has shown that every $R$-module has an RD-pure injective envelope;
see e.g. \cite[Chapter XIII, Theorem 1.6]{FS}. More recently, Divaani-Aazar, Esmkhani and
Tousi \cite[Corollary 4.7, Definition 4.8 and Theorem 4.10]{DET} showed that every $R$-module
has a cyclically pure injective envelope.

Our main aim in this paper is to prove that for any set-presentable class $\mathcal{S}$ of finitely
presented $R$-modules containing $R$, the class of  $\mathcal{S}$-pure injective $R$-modules is
enveloping. We essentially use the technique and ideas introduced by Kie{\l}pi\'nski \cite{K} and
Warfield \cite{W} and developed in \cite{EJ}, \cite{FSZ}, \cite{JS}, \cite{DET} and \cite{KS}-\cite{S2}.

First in Proposition \ref{24}, for a general class $\mathcal{S}$ of finitely presented $R$-modules, we
give a characterization of $\mathcal{S}$-pure exact sequences. Let $\mathcal{S}$ be a set-presentable
class of finitely presented $R$-modules containing $R$. In Proposition  \ref{28}, we show that the class of $\mathcal{S}$-pure injective $R$-modules is preenveloping.  This, in particular, yields that the left global $\mathcal{S}$-pure projective dimension of $R$ is equal to its left global $\mathcal{S}$-pure injective
dimension; see Corollary \ref{29}. Finally, in Theorem \ref{38}, we prove that every $R$-module has an
$\mathcal{S}$-pure injective envelope.

We continue the introduction by recalling some basic definitions and notions that we use
in this paper. Let $\mathcal{S}$ be a class of $R$-modules. An exact sequence $0\to A\overset{f}
\to B\overset{g}\to C\to 0$ of $R$-modules and $R$-homomorphisms is called $\mathcal{S}$-{\em pure}
if for all $U\in \mathcal{S}$ the induced homomorphism $\Hom_{R}(U,B)\to \Hom_{R}(U,C)$
is surjective. In this situation, $f, g, f(A)$ and $C$ are called $\mathcal{S}$-{\em pure
monomorphism}, $\mathcal{S}$-{\em pure epimorphism}, $\mathcal{S}$-{\em pure submodule} of $B$,
and $\mathcal{S}$-{\em pure homomorphic image} of $B$; respectively.  An $R$-module $P$ (resp.
$E$) is called $\mathcal{S}$-{\em pure projective} (resp. $\mathcal{S}$-{\em pure injective})
if for any $\mathcal{S}$-pure exact sequence $0\to A\overset{f}\to B\overset{g}\to C\to 0$, the
induced homomorphism $\Hom_{R}(P,B)\to \Hom_{R}(P,C)$ (resp. $\Hom_{R}(B,E)\to \Hom_{R}(A,E)$)
is surjective. Also, a right $R$-module $F$ is called $\mathcal{S}$-{\em pure flat} if for any
$\mathcal{S}$-pure exact sequence $0\to A\overset{f}\to B\overset{g}\to C\to 0$, the induced
homomorphism $F\otimes_{R} A\to F\otimes_{R}B$ is injective. An $R$-module $M$ is called
{\em cyclically-presented} if it is isomorphic to a module of the form $R^{n}/G$ for some
$n\in \BN$ and some cyclic submodule $G$ of $R^n$. If $\mathcal{S}$ is the class of all finitely
presented (resp. cyclic cyclically-presented) $R$-modules, then $\mathcal{S}$-purity is called
{\em purity} (resp. RD-{\em purity}). If $\mathcal{S}$ is the class of all cyclically-presented
$R$-modules, then $\mathcal{S}$-purity is called {\em cyclically purity}.

Let $\mathscr{X}$ be a class of $R$-modules and $M$ an $R$-module. An $R$-homomorphism $\phi:M\to
X$ where $X\in \mathscr{X}$ is called a $\mathscr{X}$-\emph{preenvelope} of $M$ if for any $X'\in
\mathscr{X}$, the induced homomorphism $\Hom_R(X,X')\to \Hom_R(M,X')$ is surjective. Also, an
$R$-homomorphism $\phi:X\to M$ where $X\in \mathscr{X}$ is called a $\mathscr{X}$-\emph{precover}
of $M$ if for any $X'\in \mathscr{X}$, the induced homomorphism $\Hom_R(X',X)\to \Hom_R(X',M)$
is surjective. If $\phi:M\to X$ (resp. $\phi:X\to M$) is a $\mathscr{X}$-preenvelope (resp.
$\mathscr{X}$-precover) of $M$ and any $R$-homomorphism $f:X\to X$ such that $f \phi=\phi$ (resp.
$\phi f=\phi$) is an automorphism, then $\phi$ is called a $\mathscr{X}$-\emph{envelope} (resp. $\mathscr{X}$-\emph{cover}) of $M$. The class $\mathscr{X}$ is called {\em (pre)enveloping} (resp.
{\em (pre)covering}) if every $R$-module admits a $\mathscr{X}$-(pre)envelope (resp.
$\mathscr{X}$-(pre)cover). By definition, it is clear that if $\mathscr{X}$-envelopes (resp.
$\mathscr{X}$-covers) exist, then they are unique up to isomorphism. Also, it is obvious that if
the class $\mathscr{X}$ contains all injective (resp. projective) $R$-modules, then any
$\mathscr{X}$-preenvelope (resp. $\mathscr{X}$-precover) is injective (resp. surjective).

\section{$\mathcal{S}$-pure exact sequences}

Propositions 2.4 and 2.8 are the main results of this section. We will use them several times
for proving our main result in the next section. One can easily deduce the following result by
the definition.

\begin{lemma}\label{21} Let $\mathcal{S}$ be a class of $R$-modules and $\{M_{\gamma}\}_{\gamma
\in \Gamma}$ a set indexed family of $R$-modules. Also, let $\{N_{\gamma}\}_{\gamma\in \Gamma}$
be a set indexed family of right $R$-modules.
\begin{enumerate}
\item[(i)] $\prod_{\gamma\in \Gamma}M_{\gamma}$ is $\mathcal{S}$-pure injective if and only if
$M_{\gamma}$ is $\mathcal{S}$-pure injective for all $\gamma\in \Gamma$.
\item[(ii)] $\bigoplus_{\gamma\in \Gamma}N_{\gamma}$ is $\mathcal{S}$-pure flat if and only if
$N_{\gamma}$ is $\mathcal{S}$-pure flat for all $\gamma\in \Gamma$.
\end{enumerate}
\end{lemma}

In what follows we denote the Pontryagin duality functor $\Hom_{\mathbb{Z}}(-,\mathbb{Q}/\mathbb{Z})$
by $(-)^{+}$.

\begin{lemma}\label{22} Let $\mathcal{S}$ be a class of $R$-modules. A right $R$-module $M$ is
$\mathcal{S} $-pure flat if and only if $M^{+}$ is $\mathcal{S}$-pure injective.
\end{lemma}

\begin{proof} Let $\mathbf{X}=0\to X_{1}\to X_{2}\to X_{3}\to 0$ be an $\mathcal{S}$-pure exact
sequence. As $\mathbb{Q}/\mathbb{Z}$ is a faithful injective $\mathbb{Z}$-module, $M\otimes_{R}
\mathbf{X}$ is exact if and only if $(M\otimes_{R}\mathbf{X} )^{+}\cong \Hom_{R}(\mathbf{X},M^{+})$
is exact. This implies the conclusion.
\end{proof}

Next, for any general class $\mathcal{S}$ of $R$-modules, we show that the class of $\mathcal{S}$-pure
flat $R$-modules is covering.

\begin{corollary}\label{23} Let $\mathcal{S}$ be a class of $R$-modules. Then every right $R$-module
admits an $\mathcal{S}$-pure flat cover.
\end{corollary}

\begin{proof} Let $M$ be a right $\mathcal{S}$-pure flat $R$-module and $0\to N\overset{f}\to M
\overset{g}\to L\to 0$ a pure exact sequence of right $R$-modules. Then, we get the split exact
sequence $0\to L^+\overset{g^+}\to M^+\overset{f^+}\to N^+\to 0$. Lemma 2.2 implies that $M^+$
is $\mathcal{S}$-pure injective, and so by Lemma 2.1 (i), one deduce that $L^+$ is
$\mathcal{S}$-pure injective. So, using Lemma 2.2 again, yields that $L$ is $\mathcal{S}$-pure
flat. Hence, the class of $\mathcal{S}$-pure flat right $R$-modules is closed under pure quotient
modules. On the other hand, by Lemma 2.1 (ii), any direct sum of $\mathcal{S}$-pure flat right
$R$-modules is $\mathcal{S}$-pure flat. Therefore, by \cite[Theorem 2.5]{HJ}, it turns out that
every right $R$-module has an $\mathcal{S}$-pure flat cover.
\end{proof}

For any two natural integers $n,k$ and any $R$-homomorphism $\mu:R^{k}\to R^{n},$ let $\mu^{t}:R^{n}
\to R^{k}$ denote the $R$-homomorphism given by the transpose of the matrix corresponding to $\mu.$
Let $U$ be a finitely presented $R$-module and $R^{k}\overset{\mu}\to R^{n}\overset{\pi}\to U\to 0$
a finitely presentation of $U$. Then, the {\em Auslander transpose} of $U$ is defined by $\tr(U):=
\Coker \mu^{t}.$ It is unique up to projective direct summands. For further information on the notion
Auslander transpose, we refer the reader to \cite[Section 11.4]{S3} and in particular to Remark in
\cite[page 185]{S3}.

The following result is an analogue of \cite[Proposition 3]{W} for a general class of finitely presented
$R$-modules; see also \cite[Lemma 1 and Theorem 1]{K}, \cite[Proposition 1.1 and Corollary 1.2]{JS}
and \cite[Lemma 4.2]{S2}.

\begin{proposition}\label{24} Let $\mathcal{S}$ be a class of finitely presented $R$-modules and
$\mathbf{E}=0\to A \hookrightarrow \hspace{-4mm}^{\vspace{-1mm}^{i}}\hspace{3mm}B\to\hspace{-5mm}^{\vspace{-1mm}^{\psi}}\hspace{3mm}C\to
0$ an exact sequence of $R$-modules and $R$-homomorphisms. The following are equivalent:
\begin{enumerate}
\item[(i)] $\mathbf{E}$ is $\mathcal{S}$-pure exact.
\item[(ii)] $\tr(U)\otimes_{R}\mathbf{E}$ is exact for all $U\in \mathcal{S}.$
\item[(iii)] $\mu(A^{k})=A^{n}\cap\mu(B^{k})$ for all matrices $\mu \in \Hom_{R}(R^{k},R^{n})$ with
$\Coker \mu^{t} \in \mathcal{S}$.
\item[(iv)] for any matrix $(r_{ij})\in \Hom_{R}(R^{n},R^{k})$ with $\Coker (r_{ij})\in \mathcal{S}$
and any $a_1,\ldots, a_n\in A$ if the linear equations $\sum^{k}_{i=1} r_{ij}x_{i}=a_{j}; 1\leqslant
j\leqslant n$ are soluble in $B$, then they are also soluble in $A$.
\end{enumerate}
\end{proposition}

\begin{proof} (i)$\Rightarrow$(iv) Let $(r_{ij})\in \Hom_{R}(R^{n},R^{k})$ be a matrix with $\Coker (r_{ij})
\in \mathcal{S}$ and $a_1,\ldots, a_n\in A$. Set $U:=\Coker (r_{ij})$. Then $U$ has generators $u_1,\ldots,
u_k$ which satisfy the relations $\sum^{k}_{i=1} r_{ij}u_{i}=0;$  $1\leqslant j\leqslant n$. Assume that
the linear equations $$\sum^{k}_{i=1}r_{ij}x_{i}=a_{j};  \    \   1\leqslant j\leqslant n$$ are soluble in
$B$. We show that they are also soluble in $A$. Let $y_{1},\ldots, y_{k}\in B$ be a solution of these
equations. The map $f\in \Hom_{R}(U,C)$ given by $f(u_{i}):=\psi(y_{i})$ for all $1\leqslant i\leqslant k$
is a well-defined $R$-homomorphism. As $\mathbf{E}$ is $\mathcal{S}$-pure exact, the induced homomorphism $\Hom_{R}(U,B)\to \Hom_{R}(U,C)$ is surjective, and so there exists an $R$-homomorphism
$g\in \Hom_{R}(U,B)$ such that $f=\psi g$. Let $z_{i}:=y_{i}-g(u_{i})$ for all $i=1,\ldots, k.$
Then each $z_{i}$ belongs to $\Ker \psi=A$ and $\sum^{k}_{i=1}r_{ij}z_{i}=a_{j}$ for all $j=1,\ldots, n.$

(iv)$\Rightarrow$(i) Let $U$ be an element of $\mathcal{S}$ which is generated by elements $u_{1},\ldots,
u_{k}$ which satisfy the relations $\sum^{k}_{i=1}r_{ij}u_{i}=0;  1\leqslant j\leqslant n$. Let $f\in
\Hom_{R}(U,C)$. For each $i=1,\ldots, k,$ choose $y_{i}\in B$ such that $\psi(y_{i})=f(u_{i}).$ Then
$\sum^{k}_{i=1}r_{ij}y_{i}\in \Ker \psi=A$ for all $j=1,\ldots, n.$ Therefore, we have a set of linear
equations: $$\sum^{k}_{i=1}r_{ij}x_{i}=a_{j}; \   \   1\leqslant j\leqslant n$$ with constants
in $A$ which are soluble in $B$. Let $z_{1},\ldots, z_{k}$ be a solution of these equations in $A$. We
define $g\in \Hom_{R}(U,B)$ by $g(u_{i}):=y_{i}-z_{i}$ for all $i=1,\ldots, k.$ Then $\psi g=f,$ and so
the induced homomorphism $\Hom_{R}(U,B)\to \Hom_{R}(U,C)$ is surjective.

(ii)$\Leftrightarrow$(iii) Let $\mu=(r_{ij})\in \Hom_{R}(R^{k},R^{n})$ be a matrix with
$U:=\Coker \mu^t\in \mathcal{S}.$ Tensoring the exact sequence $R^k\overset{\mu}\to R^n
\overset{\pi}\to \tr(U)\to 0,$ by $A$ first and then by $B$ yield the following commutative diagram:
$$\xymatrix{A^{k}\ar[d]_{i^k}\ar[r]^{\mu}&A^{n}\ar[d]_{i^n}\ar[r]^{\hspace{-4mm}\pi_{A}}
& \tr(U)\otimes_{R}A \ar[d]^{1_{\tr(U)}\underset{R}\otimes i}\ar[r]&0 \\
B^k\ar[r]_{\mu}&B^n\ar[r]_{\hspace{-4mm}\pi_{B}}&\tr(U)\otimes_{R}B\ar[r]&0 }$$\\
in which all maps are natural, rows are exact and the left and middle vertical maps are injective. Clearly, $1_{\tr(U)}\otimes_{R}i$ is injective if and only if $\Ker \pi_{A}=\Ker((1_{\tr(U)}\otimes_{R}i)(\pi_{A}))$.
On the other hand, we have:
$$\begin{array}{llll}
\Ker((1_{\tr(U)}\otimes_{R}i)(\pi_{A}))&=\Ker (\pi_{B}i^n)\\
&=A^{n}\cap \Ker \pi_{B}\\
&=A^{n}\cap \mu(B^{k}).
\end{array}$$
Therefore, $\tr(U)\otimes_{R}\mathbf{E}$ is exact if and only if $\mu(A^{k})=A^{n}\cap\mu(B^{k}).$

(iii)$\Rightarrow$(iv) Assume that $(r_{ij})\in \Hom_{R}(R^{n},R^{k})$ be a matrix with $\Coker
(r_{ij})\in \mathcal{S}$. Consider the linear equations: $$\sum^{k}_{i=1}r_{ij}x_{i}=a_{j};  \   \
1\leqslant j\leqslant n$$ with constants in $A$. Let $b_{1},\ldots, b_{k}$ be a solution of these
equations in $B$.  Set $\mu:=(r_{ij})^t$. Then, the hypothesis yields that $\mu(A^{k})=A^{n}\cap
\mu(B^{k}).$ As $(a_{1},\ldots, a_{n})\in A^{n}\cap \mu(B^{k})$, there exists $(\acute{a_{1}},
\ldots, \acute{a_{k}})\in A^{k}$ such that $\mu((\acute{a_{1}},\ldots, \acute{a_{k}}))=(a_{1},\ldots,
a_{n})$. Consequently, $\acute{a_{1}},\ldots, \acute{a_{k}}$ is a solution of the above equations in
$A$.

(iv)$\Rightarrow$(iii) Let $\mu=(r_{ij})\in \Hom_{R}(R^{k},R^{n})$ be a matrix with $U:=\Coker \mu^{t}
\in \mathcal{S}$. Let $(a_{1},\ldots,a_{n})\in A^{n}\cap \mu(B^{k})$. Then $\mu(b_{1},\ldots,b_{k})=
(a_{1},\ldots, a_{n})$ for some $b_{1},\ldots, b_{k}\in B$. Hence, $b_{1},\ldots, b_{k}$ is a solution
of the equations $$\sum^{k}_{i=1}r_{ji}x_{i}=a_{j}; \   \   1\leqslant j\leqslant n.$$  Let $\acute{a_{1}},
\ldots, \acute{a_{k}}\in A$ be a solution of the above equations. Then $\mu(\acute{a_{1}},\ldots,\acute{a_{k}})=(a_{1},\ldots, a_{n}),$ and so $(a_{1},\ldots,a_{n})\in
\mu(A^{k})$.
\end{proof}

Now, we deduce a couple of corollaries of Proposition \ref{24}.

\begin{corollary}\label{25} Let $\mathcal{S}$ be a class of finitely presented $R$-modules and $\mathbf{X}
=0\to X_{1}\to X_{2}\to X_{3}\to 0$ an exact sequence of $R$-modules and $R$-homomorphisms. Then the
following conditions are equivalent:
\begin{enumerate}
\item[(i)] $\mathbf{X}$ is $\mathcal{S}$-pure exact.
\item[(ii)] $\Hom_{R}(P,\mathbf{X})$ is exact for all $\mathcal{S}$-pure projective $R$-modules $P.$
\item[(iii)] $\Hom_{R}(\mathbf{X},E)$ is exact for all $\mathcal{S}$-pure injective $R$-modules $E.$
\item[(iv)] $F\otimes_{R} \mathbf{X}$ is exact for all $\mathcal{S}$-pure flat $R$-modules $F.$
\end{enumerate}
\end{corollary}

\begin{proof} (i)$\Rightarrow$(ii) and (i)$\Rightarrow$(iii) are clear. (ii)$\Rightarrow$(i) comes from the
fact that every $U\in \mathcal{S}$ is $\mathcal{S}$-pure projective.

(iii)$\Rightarrow$(iv) Let $F$ be an $\mathcal{S}$-pure flat $R$-module. Then, by Lemma \ref{22}, $F^{+}$
is $\mathcal{S}$-pure injective. So, $$\Hom_{R}(\mathbf{X},F^{+})\cong \Hom_{\mathbb{Z}}(F\otimes_{R} \mathbf{X},\mathbb{Q}/\mathbb{Z})=(F\otimes_{R} \mathbf{X})^{+}$$ is exact. Since $\mathbb{Q}/\mathbb{Z}$
is a faithful injective $\mathbb{Z}$-module, it follows that $F\otimes_{R}\mathbf{X}$ is exact.

(iv)$\Rightarrow$(i) By Proposition  \ref{24}, $\tr(U)$ is $\mathcal{S}$-pure flat for all $U\in \mathcal{S}$.
Hence, Proposition  \ref{24} implies that the sequence $\mathbf{X}$ is $\mathcal{S}$-pure exact.
\end{proof}

In what follows, for a class $\mathcal{S}$ of finitely presented $R$-modules, we denote the  class
$\{\tr(U)|U\in \mathcal{S}\}$ by $\tr(\mathcal{S})$.

\begin{corollary}\label{26} Assume that $R$ is commutative and $\mathcal{S}$ is a set-presentable class
of finitely presented $R$-modules containing $R$. If $\mathcal{S}\subseteq \tr(\mathcal{S})$, then every $\mathcal{S}$-pure projective $R$-module is $\mathcal{S}$-pure flat.
\end{corollary}

\begin{proof} Assume that $\mathcal{S}\subseteq \tr(\mathcal{S})$. Then, by Proposition  \ref{24}, any element
of $\mathcal{S}$ is $\mathcal{S}$-pure flat. By \cite[Proposition 1]{W}, an $R$-module $M$ is
$\mathcal{S}$-pure projective if and only if it is a summand of a direct sum of copies of modules
in $\mathcal{S}$. Thus, by Lemma 2.1 (ii), every $\mathcal{S}$-pure projective $R$-module is
$\mathcal{S}$-pure flat.
\end{proof}

\begin{example}\label{27} Let $\mathcal{S}$ be a class of finitely presented $R$-modules.
\begin{enumerate}
\item[(i)] If $\mathcal{S}$ is the class of all cyclic free $R$-modules, then $\mathcal{S}$-pure exact
sequences are the usual exact sequences. So, $\mathcal{S}$-pure projective, $\mathcal{S}$-pure injective and $\mathcal{S}$-pure flat $R$-modules are the usual projective, injective and flat $R$-modules; respectively.
\item[(ii)] If $\mathcal{S}$ is the class of all finitely presented $R$-modules, then
$\mathcal{S}$-purity coincides with the usual purity.
\item[(iii)] If $\mathcal{S}$ is the class of all cyclic cyclically-presented $R$-modules,
then $\mathcal{S}$-purity coincides with the RD-purity.
\item[(iv)] If $\mathcal{S}$ is the class of all cyclically-presented $R$-modules, then
$\mathcal{S}$-purity coincides with the cyclically purity.
\item[(v)] Assume that $R$ is commutative. Obliviously, if $R\in \mathcal{S}$, then $R\in \tr(\mathcal{S})$.
It is easy to see that if $\mathcal{S}$ is set-presentable, then $\tr(\mathcal{S})$ has a subclass $\mathcal{\widetilde{S}}$, which is a set and $\tr(\mathcal{S})$-purity coincides with $\mathcal{\widetilde{S}}$-purity.  In the cases (i), (ii) and (iii) above, one can easily verify that $\mathcal{S}=\tr(\mathcal{S})$. In case (iv), $\tr(\mathcal{S})$-purity coincides with $\mathcal{\widetilde{S}}$-purity, where $\mathcal{\widetilde{S}}$
is the set $$\{R/I| I \   \  \text{is a finitely generated ideal of} \   \  R\}.$$
\end{enumerate}
\end{example}

\begin{proposition}\label{28} Let $\mathcal{S}$ be a set-presentable class of finitely presented $R$-modules
containing $R$. Then every $R$-module $M$ admits an $\mathcal{S}$-pure injective preenvelope.
\end{proposition}

\begin{proof} Since $\mathcal{S}$ is set-presentable, it has a subclass $\mathcal{S}^*$, which is a set,
with the property that for any $U\in \mathcal{S}$ there is a $U^*\in \mathcal{S}^*$ with $U^*\cong U.$
Let $\Gamma$ be the set of all pairs $(U,f)$ with $U\in \mathcal{S}^*$ and $f\in \Hom_R(M,\tr(U)^+)$,
and for each $\gamma\in \Gamma$ denote the corresponding $U$ and $f$ by $U_\gamma$ and $f_
\gamma$. Let $E:=\prod_{\gamma\in \Gamma}\tr(U_\gamma)^+$ and let $\phi:M\to E$ be an $R$-homomorphism
defined by $\phi(x)=(f_\gamma(x))_\gamma$. Then, by Proposition 2.4, Lemma 2.2 and Lemma 2.1 (i),
it follows that $E$ is an $\mathcal{S}$-pure injective $R$-module. As $R\in \mathcal{S}$, it is easy to
see that $\phi$ is injective. We show that $\phi$ is our desired $\mathcal{S}$-pure injective preenvelope.
To this end, by Corollary \ref{25}, it is enough to check that $\phi$ is an $\mathcal{S}$-pure monomorphism.
For any $U\in \mathcal{S}^*,$ the homomorphism $$1_{\tr(U)}\otimes \phi:\tr(U)\otimes_RM\to \tr(U)
\otimes_RE$$ is injective if and only if $$(1_{\tr(U)}\otimes \phi)^+:(\tr(U)\otimes_RE)^+\to (\tr(U)
\otimes_RM)^+$$ is surjective. Now, consider the following commutative diagram:
$$\xymatrix{(\tr(U)\otimes_RE)^+\  \ar[d]_{\cong}\ar[r]^{(1_{\tr(U)}\otimes \phi)^+}& (\tr(U)\otimes_RM)^+
\ar[d]^{\cong} \\ \Hom_R(E,\tr(U)^+)\hspace{4mm}\ \ar[r]^{\Hom_R(\phi,1_{\tr(U)^+})}& \hspace{6mm}\Hom_R(M,
\tr(U)^+).}$$
Since, the vertical maps are isomorphisms and, by our construction, the bottom map is surjective, we
deduce that $1_{\tr(U)}\otimes \phi$ is injective. Thus, by Proposition  \ref{24}, it turns out that $\phi$
is an $\mathcal{S}$-pure monomorphism.
\end{proof}

Let $\mathcal{F}$ and  $\mathcal{G}$ be two classes of $R$-modules.  The functor $\Hom_R(-,\sim)$ is
said to be {\em right balanced} by $\mathcal{F}\times\mathcal{G}$ if for any $R$-modules $M$, there are
complexes $${\bf F_{\bullet}}=\cdots \to F_n\to F_{n-1}\to \cdots \to F_0\to M\to 0$$ and
$${\bf G^{\bullet}}=0\to M\to G^0\to \cdots \to G^n\to G^{n+1}\to \cdots$$ in which $F_n\in \mathcal{F},
G^n\in \mathcal{G}$ for all $n\geq 0,$ such that for any $F\in \mathcal{F}$ and any $G\in \mathcal{G}$,
the two complexes $\Hom_R({\bf F_{\bullet}},G)$ and $\Hom_R(F,{\bf G^{\bullet}})$ are exact.

The concept of pure homological dimensions was introduced in a special case by Griffith in \cite{G}, and
in a general setting by  Kie{\l}pi\'nski and Simson in \cite{KS}. For an $R$-module $M$, we define $\mathcal{S}$-\emph{pure projective dimension} of $M$ as the infimum of the length of left $\mathcal{S}$-pure
exact resolutions of $M$ which are consisting of $\mathcal{S}$-pure projective $R$-modules. Then \emph{left
global $\mathcal{S}$-pure projective dimension} of $R$ is defined to be the supremum of $\mathcal{S}$-pure projective dimensions of all $R$-modules. $\mathcal{S}$-\emph{pure injective dimension} of  $R$-modules and \emph{left global
$\mathcal{S}$-pure injective dimension} of $R$ are defined dually.

We end this section by recording the following useful application.

\begin{corollary}\label{29} Let $\mathcal{S}$ be a set-presentable class of finitely presented $R$-modules
containing $R$. Denote the class of all $\mathcal{S}$-pure projective (resp. $\mathcal{S}$-pure injective)
$R$-modules by $\mathcal{SP}$ (resp. $\mathcal{SI}$). Then the functor $\Hom_R(-,\sim)$ is right balanced
by $\mathcal{SP}\times\mathcal{SI}$. In particular, the left global $\mathcal{S}$-pure projective dimension of
$R$ is equal to its left global $\mathcal{S}$-pure injective dimension.
\end{corollary}

\begin{proof} Let $M$ and $N$ be two $R$-modules. In view of \cite[Proposition 1]{W} and Corollary
\ref{25}, we can construct an exact complex $${\bf P_{\bullet}}=\cdots \to P_n\to P_{n-1}\to \cdots
\to P_0\to M\to 0$$ such that each $P_n$ is $\mathcal{S}$-pure projective, and for any
$\mathcal{S}$-pure projective $R$-module $P$ and any $\mathcal{S}$-pure injective $R$-module $I$,
the two complexes $\Hom_R(P,{\bf P_{\bullet}})$ and $\Hom_R({\bf P_{\bullet}},I)$ are exact. Also,
by Proposition  \ref{28} and Corollary \ref{25}, we can construct an exact complex $${\bf I^{\bullet}}
=0\to N\to I^0\to \cdots \to I^n\to I^{n+1}\to \cdots$$ such that each $I^n$ is $\mathcal{S}$-pure
injective, and for any $\mathcal{S}$-pure injective $R$-module $I$ and any $\mathcal{S}$-pure
projective $R$-module $P,$ the two complexes $\Hom_R({\bf I^{\bullet}},I)$ and $\Hom_R(P,I^{\bullet})$
are exact. Thus, $\Hom_R(-,\sim)$ is right balanced  by $\mathcal{SP}\times\mathcal{SI}$.

Denote the complexes $$\cdots \to P_n\to P_{n-1}\to \cdots \to P_0\to 0$$  and $$0\to I^0\to \cdots \to
I^n\to I^{n+1}\to \cdots$$ by ${\bf P_{\circ}}$ and ${\bf I^{\circ}}$, respectively. Then \cite[Theorem 8.2.14]{EJ} yields that the complexes $\Hom_R({\bf P_{\circ}},N)$ and $\Hom_R(M,{\bf
I^{\circ}})$ have isomorphic homology modules. Let $n$ be a non-negative integer. In view of
\cite[Theorem 8.2.3 (2) and Corollary 8.2.4 (2)]{EJ}, it is straightforward to check that the pure projective
dimension of $M$ is less than or equal $n$
if and only if $H^{n+1}(\Hom_R({\bf P_{\circ}},L))=0$ for all $R$-modules $L$. Also, by \cite[Theorem
8.2.5 (1) and Corollary 8.2.6 (1)]{EJ}, we can deduce that the pure injective dimension of $N$ is less
than or equal $n$ if and only if $H^{n+1}(\Hom_R(L,{\bf I^{\circ}}))=0$ for all $R$-modules $L$. These
facts yield that the left global $\mathcal{S}$-pure projective dimension of $R$ is equal to its left
global $\mathcal{S}$-pure injective dimension.
\end{proof}

\section{$\mathcal{S}$-pure injective envelops}

For proving Theorem \ref{38}, which is our main result, we need to prove the following five preliminary
lemmas. We begin this section with the following definition (compare with \cite{EJ}, \cite{FSZ}, \cite{K},
\cite{KS}, \cite{S2} and \cite{W}).

\begin{definition}\label{31} Let $\mathcal{S}$ be a class of $R$-modules and $N$ an $\mathcal{S}$-pure
submodule of an $R$-module $M$.
\begin{enumerate}
\item[(i)] We say $M$ is an $\mathcal{S}$-{\em pure essential extension} of $N$ if any $R$-homomorphism $\varphi:M\to L$ with $\varphi|_{N}$ $\mathcal{S}$-pure monomorphism, is injective.
\item[(ii)]  We say $M$ is a {\em maximal $\mathcal{S}$-pure essential extension} of $N$ if $M$ is
an $\mathcal{S}$-pure essential extension of $N$ and no proper extension of $M$ is an $\mathcal{S}$-pure
essential extension of $N$.
\item[(iii)]  We say $M$ is a {\em minimal $\mathcal{S}$-pure injective extension} of $N$ if $M$
is $\mathcal{S}$-pure injective and no proper $\mathcal{S}$-pure injective submodule of $M$ contains $N$.
\end{enumerate}
\end{definition}

\begin{lemma}\label{32} Let $\mathcal{S}$ be a class of $R$-modules. Let $M$ and $M^{\prime}$ be two
$R$-modules and $f:M \to M^{\prime}$ an $R$-isomorphism. Let $N$ be a submodule of $M$ and $N^{\prime}:
=f(N).$
\begin{enumerate}
\item[(i)] $N$ is an $\mathcal{S}$-pure submodule of $M$ if and only if $N^{\prime}$ is an $\mathcal{S}$-pure
submodule of $M^{\prime}$.
\item[(ii)] $M$ is an $\mathcal{S}$-pure essential extension of $N$ if and only if $M^{\prime}$
is an $\mathcal{S}$-pure essential extension of $N^{\prime}$.
\item[(iii)] $M$ is a maximal $\mathcal{S}$-pure essential extension of $N$ if and only if
$M^{\prime}$ is a maximal $\mathcal{S}$-pure essential extension of $N^{\prime}$.
\end{enumerate}
\end{lemma}

\begin{proof} (i) is clear.

(ii) Assume that $M$ is an $\mathcal{S}$-pure essential extension of $N$. By (i), $N^{\prime}$ is
an $\mathcal{S}$-pure submodule of $M^{\prime}$. Let $\varphi:M^{\prime}\to L$ be an $R$-homomorphism
such that $\varphi|_{N^{\prime}}$ is an $\mathcal{S}$-pure monomorphism. Then $\varphi f:M\to L$ is
an $R$-homomorphism such that $(\varphi f)|_{N}$ is an $\mathcal{S}$-pure monomorphism. Now, as $M$
is an $\mathcal{S}$-pure essential extension of $N$, it follows that $\varphi f$ is injective, and so
$\varphi$ is also injective. The converse follows by the symmetry. Note that $f^{-1}:M^{\prime}\to M$
is an $R$-isomorphism with $f^{-1}(N^{\prime})=N.$

(iii) By the symmetry, it is enough to show the ``only if" part. Suppose that M is a maximal
$\mathcal{S}$-pure essential extension of $N$. By (ii), $M^{\prime}$ is an $\mathcal{S}$-pure
essential extension of $N^{\prime}$. Let $L^{\prime}$ be an extension of $M^{\prime}$ which is an
$\mathcal{S}$-pure essential extension of $N^{\prime}$. By \cite[Proposition 1.1]{SV},
there are an extension $L$ of $M$ and an $R$-isomorphism $g:L\to L^{\prime}$ such that the following
diagram commutes:
$$\xymatrix{N \ \ar[d]_{f|_{N}} \ar@{^{(}->}[r] &M\  \ar[d]_{f}\ar@{^{(}->}[r] &L\ar[d]^{g} \\
N^{\prime}\ \ar[r]\ar@{^{(}->}[r] &M^{\prime}\ \ar@{^{(}->}[r]& L^{\prime} }$$\\
It follows by (ii), that $L$ is an $\mathcal{S}$-pure essential extension of $N$. Hence,
by the maximality assumption on $M$, we obtain that $L=M$. Therefore $L^{\prime}=M^{\prime},$ as required.
\end{proof}

\begin{lemma}\label{33} Let $\mathcal{S}$ be a class of finitely presented $R$-modules and $N$ an
$\mathcal{S}$-pure submodule of an $R$-module $M$. Then there exists a submodule $K$ of $M$ such
that $K\cap N=0$ and $M/K$ is an $\mathcal{S}$-pure essential extension of $(K+N)/K$.
\end{lemma}

\begin{proof} Let $\Sigma$ denote the set of all submodules $U$ of $M$ which satisfy the following
conditions:
\begin{enumerate}
\item[(i)] $U\cap N=0;$ and
\item[(ii)] $(U+N)/U$ is an $\mathcal{S}$-pure submodule of $M/U$.
\end{enumerate}
Then $\Sigma$ is not empty, because $0\in \Sigma$. Let $\{K_{\alpha}\}_{\alpha\in \Omega}$ be a totally
ordered subset of $\Sigma$ and set $\widetilde{K}:=\cup _{\alpha\in \Omega} K_{\alpha}.$ We show that
$\widetilde{K}$ satisfies the conditions (i) and (ii).  Clearly, $\widetilde{K}\cap N=0.$ Let $(r_{ij})
\in \Hom_{R}(R^{n},R^{k})$ be a matrix with $\Coker (r_{ij})\in \mathcal{S}.$ Let $$\sum_{i=1}^{k}r_{ij} x_{i}=a_{j}+\widetilde{K}; \     \  1\leqslant j\leqslant n  \    \ (*)$$ be a set of linear equations
with constants in $(\widetilde{K}+N)/\widetilde{K}$. Let $y_{1}+\widetilde{K},\ldots, y_{k}+
\widetilde{K}$ be a solution of these equations in $M/\widetilde{K}$. Then $\sum^{k}_{i=1} r_{ij}y_{i}
-a_{j}\in \widetilde{K}$ for all $j=1,\ldots, n.$ There exists $\beta\in \Omega$ such that $\sum^{k}_{i=1}r_{ij}y_{i}-a_{j}\in K_{\beta}$ for all $j=1,\ldots, n.$  So, $y_{1}+K_{\beta},\ldots,
y_{k}+K_{\beta}$ is a solution of the equations $\sum_{i=1}^{k}r_{ij} x_{i}=a_{j}+K_{\beta};$ $1\leqslant
j\leqslant n$ in $M/K_{\beta}.$ Now, as $(K_{\beta}+N)/K_{\beta}$ is an $\mathcal{S}$-pure submodule of
$M/K_{\beta},$ there exists $z_{1},\ldots, z_{k}\in N$ such that $$\sum_{i=1}^{k}r_{ij}z_{i}-a_{j}\in K_{\beta}\subseteq \widetilde{K}$$ for all $j=1,\ldots, n.$ Hence, $z_{1}+\widetilde{K},\ldots, z_{k}+
\widetilde{K}$ is a solution of the equations $(*)$ in $(\widetilde{K}+N)/\widetilde{K}$. So, by
Proposition  \ref{24}, $(\widetilde{K}+N)/\widetilde{K}$ is an $\mathcal{S}$-pure submodule of
$M/\widetilde{K}$. Thus, by Zorn's Lemma, $\Sigma$ has a maximal element $K$.

Suppose that $\varphi:M/K\to L$ is an $R$-homomorphism such that the restriction $\varphi|_{(N+K)/K}$
is an $\mathcal{S}$-pure monomorphism. Let $\Ker \varphi=K^{\prime}/K.$ Then $\varphi$ induces an
$R$-monomorphism $$\varphi^{*}:(M/K)/(K^{\prime}/K)\to L.$$ Set $P:=((N+K)/K+K^{\prime}/K)/(K^{\prime}/K)$.
Since $\varphi((N+K)/K)$ is an $\mathcal{S}$-pure submodule of $L$ and
$$\begin{array}{ccc}
\varphi((N+K)/K) & = &\hspace{-2.5cm}\varphi^{*}(P) \\
&\leqslant & \varphi^{*}((M/K)/(K^{\prime}/K)) \\
&\leqslant & \hspace{-3cm}L,
\end{array}$$
it follows that $\varphi^{*}(P)$ is an $\mathcal{S}$-pure submodule of $\varphi^{*}((M/K)/(K^{\prime}/K)),$
and so by Lemma 3.2 (i), we conclude that $(N+K^{\prime})/K^{\prime}$ is an $\mathcal{S}$-pure submodule of
$M/K^{\prime}.$ Now, $K^{\prime}$ is a submodule of $M$ containing $K$ satisfying the condition $(ii).$ We
can easily check that $K^{\prime}$ also satisfies the condition $(i)$; i.e. $K^{\prime}\cap N=0$. Hence, by
the maximality of $K$, we obtain that $K^{\prime}=K,$ and so  $\varphi$ is injective.
\end{proof}

Next, as an application of the above lemma, we present a characterization of $\mathcal{S}$-pure injective
$R$-modules.

\begin{corollary}\label{34}  Let $\mathcal{S}$ be a set-presentable class of finitely presented $R$-modules
containing $R$. Then for an $R$-module $E$, the following are equivalent:
\begin{enumerate}
\item[(i)] $E$ is $\mathcal{S}$-pure injective.
\item[(ii)] $E$ has no proper $\mathcal{S}$-pure essential extension.
\end{enumerate}
\end{corollary}

\begin{proof} (i)$\Rightarrow$(ii) Let $M$ be an $\mathcal{S}$-pure essential extension of $E$.
Then $0\to E\overset{i}\hookrightarrow M\to M/E\to 0$ is an $\mathcal{S}$-pure exact sequence.
Since $E$ is $\mathcal{S}$-pure injective, there is an $R$-homomorphism $f:M\to E$ such that
$fi=1_E$. Then $M=E+\Ker f$ and $E\cap \Ker f=0$. Denote the $R$-homomorphism $if:M\to M$ by
$\varphi$. Then $\varphi|_{E}=i$, and so $\varphi|_{E}$ is an $\mathcal{S}$-pure monomorphism.
Hence $\varphi$ is injective, because $M$ is an $\mathcal{S}$-pure essential extension of $E$.
This implies that $\Ker f=\Ker \varphi=0,$ and so $M=E.$

(ii)$\Rightarrow$(i) By Proposition  \ref{28}, there exists an $\mathcal{S}$-pure injective
extension $L$ of $E$. By Lemma \ref{33}, there is a submodule $K$ of $L$ such that $L/K$ is
an $\mathcal{S} $-pure essential extension of $(E+K)/K$ and $E\cap K=0.$ But, $E$ has no proper
$\mathcal{S} $-pure essential extension, and so $E+K=L.$  This implies that $L=E\oplus K.$
Now, by Lemma \ref{21} (i), we deduce that $E$ is $\mathcal{S}$-pure injective.
\end{proof}

\begin{lemma}\label{35} Let $\mathcal{S}$ be a class of finitely presented $R$-modules, $E$ an
$\mathcal{S}$-pure injective $R$-module and $N$ an $\mathcal{S}$-pure submodule of $E$. There
is a submodule $M$ of $E$ which is a maximal $\mathcal{S}$-pure essential extension of $N$.
\end{lemma}

\begin{proof} Denote the inclusion map $N\hookrightarrow E$ by $i$. Let $L$ be an $\mathcal{S}$-pure
essential extension of $N$. Since $E$ is an $\mathcal{S}$-pure injective
$R$-module and $i$ is an $\mathcal{S}$-pure monomorphism, there exists an $R$-monomorphism
$\psi:L\to E$ such that $\psi|_N=i$. So, one has $|L|=|\psi(L)|\leq |E|$. If $L$ is a maximal
$\mathcal{S}$-pure essential extension of $N,$ then by Lemma \ref{32} (iii), $\psi(L)$ is also
a maximal $\mathcal{S}$-pure essential extension of $N$. Hence, the proof will be completed if
we show that $N$ has a maximal $\mathcal{S}$-pure essential extension. Suppose that the contrary
is true. Then, by using transfinite induction, we show that for any ordinal $\beta$, there is an
$\mathcal{S}$-pure essential extension $M_\beta$ of $N$. Set $M_0:=N$. Let $\beta$ be an ordinal
and assume that $M_\alpha$ is defined for all $\alpha<\beta$. Assume that $\beta$ has a predecessor
$\beta-1$. As $M_{\beta-1}$ is not a maximal $\mathcal{S}$-pure essential extension of $N$, there
is a proper extension $M_{\beta}$ of $M_{\beta-1}$ such that $M_{\beta}$ is an $\mathcal{S}$-pure
essential extension of $N$. If $\beta$ is a limit ordinal, then in view of Proposition  \ref{24},
it is easy to see that $M_\beta:=\cup_{\alpha<\beta}M_\alpha$ is an $\mathcal{S}$-pure essential
extension of $N$. Let $\beta$ be an ordinal with $|\beta|>|E|$. Then, $|\beta|\leq |M_\beta|\leq |E|$,
which is a contradiction.
\end{proof}

\begin{lemma}\label{36} Let $\mathcal{S}$ be a set-presentable class of finitely presented
$R$-modules containing $R$. Let $M$ be an $R$-module and $E$ a maximal $\mathcal{S}$-pure
essential extension of $M$. Then $E$ is an $\mathcal{S}$-pure injective $R$-module.
\end{lemma}

\begin{proof} In view of Proposition  \ref{28}, Corollary \ref{25} and Lemma \ref{21} (i), it is enough to
show that $E$ is a direct summand of every $R$-module which contains $E$ as an $\mathcal{S}$-pure
submodule.

Let $E$ be an $\mathcal{S}$-pure submodule of an $R$-module $L$. Since, $L$ is also an $\mathcal{S}$-pure
extension of $M$, by Lemma \ref{33}, there exists a submodule $K$ of $L$ such that $K\cap M=0$ and that
$L/K$ is an $\mathcal{S}$-pure essential extension of $(K+M)/K.$ We show that $L\cong K\oplus E$. First,
we show that $K_{1}:=K\cap E=0.$  Let $\pi:E\to E/K_1$ denote the natural epimorphism. As $K_{1}\cap M=0$,
we see that $\pi|_M$ is an $\mathcal{S}$-pure monomorphism. Hence, $\pi$ is injective, and so $K_{1}=0.$
Now, let $f:E\to(K + E)/K$ denote the natural isomorphism. Then $f (M) = (K +M)/K.$ Thus, by Lemma \ref{32}
(iii), we have $(K+E)/K$ is a maximal $\mathcal{S}$-pure essential extension of $(K +M)/K.$ But, $L/K$ is an $\mathcal{S}$-pure essential extension of $(K + M)/K$ and $(K +E)/K\subseteq L/K.$ Thus $L=K+E,$ and so
$L=K\oplus E,$ as required.
\end{proof}

\begin{lemma}\label{37} Let $\mathcal{S}$ be a set-presentable class of finitely presented $R$-modules
containing $R$. Let $E$ be an $R$-module and $M$ a submodule of $E$. The following are equivalent:
\begin{enumerate}
\item[(i)] $E$ is a maximal $\mathcal{S}$-pure essential extension of $M$.
\item[(ii)] $E$ is an $\mathcal{S}$-pure essential extension of $M$ which is $\mathcal{S}$-pure injective.
\item[(iii)] $E$ is a minimal $\mathcal{S}$-pure injective extension of $M$.
\end{enumerate}
\end{lemma}

\begin{proof} (i)$\Rightarrow$(ii) is clear by Lemma \ref{36}.

(ii)$\Rightarrow$(iii) Suppose $E_{1}$ is a submodule of $E$ containing $M$ such that $E_{1}$ is
$\mathcal{S}$-pure injective. By Lemma \ref{35}, there exists a submodule $E_{2}$ of $E_{1}$ which
is a maximal $\mathcal{S}$-pure essential extension of $M$. Since $E$ is an $\mathcal{S}$-pure essential
extension of $M$, it turns out that $E_{2}=E$. Hence $E_{1}=E.$

(iii)$\Rightarrow$(i) By Lemma \ref{35}, there is a submodule $E_{1}$ of $E$ such that $E_{1}$ is
a maximal $\mathcal{S}$-pure essential extension of $M$. Now, Lemma \ref{36} yields that $E_{1}$ is
$\mathcal{S}$-pure injective. Thus, by the minimality assumption on $E$, we get $E_{1}=E$.
\end{proof}

Finally, we are ready to prove our main result.

\begin{theorem}\label{38} Let $\mathcal{S}$ be a set-presentable class of finitely presented $R$-modules
containing $R$. Then every $R$-module $M$ possesses an $\mathcal{S}$-pure injective envelope.
\end{theorem}

\begin{proof} By Proposition \ref{28} and Lemma \ref{35}, there exists a maximal $\mathcal{S}$-pure essential
extension $E$ of $M$. Let $\phi:M\hookrightarrow E$ denote the inclusion $R$-homomorphism. Let $E^{\prime}$
be an $\mathcal{S}$-pure injective $R$-module and $\psi:M\to E^{\prime}$ an $R$-homomorphism. Since
$E^{\prime}$ is $\mathcal{S}$-pure injective, there exists an $R$-homomorphism $f:E\to E^{\prime}$
such that $f\phi=\psi.$ Now, suppose an $R$-homomorphism $f:E\to E$ is such that $f\phi=\phi.$ Since
$f|_M=\phi$ is an $\mathcal{S}$-pure monomorphism and $E$ is an $\mathcal{S}$-pure essential extension
of $M$, we see that $f$ is injective. By Lemma \ref{32} (iii), $f(E)$ is also a maximal $\mathcal{S}$-pure
essential extension of $M$. Hence, by Lemma \ref{36}, $f(E)$ is $\mathcal{S}$-pure injective. Now, as by
Lemma \ref{37}, $E$ is a minimal $\mathcal{S}$-pure injective extension of $M$, we deduce that $f(E)=E.$
So, $f$ is an  automorphism.
\end{proof}

\subsection*{Acknowledgements} The authors would like to express their deep thanks to Professor Daniel Simson 
for several useful suggestions, historical comments, and his interest on this work. In particular, they thank
him for pointing out that the results of the present paper are essentially related with the old papers by
Cohn \cite{C},  {\L}o\'{s} \cite{L}, and Kie{\l}pi\'nsski \cite{K}.


\end{document}